\documentclass[12pt]{amsart}
\usepackage{enumerate}
\usepackage{amsmath}
\usepackage{mathrsfs}
\usepackage{color}
\usepackage{tikz}
\usepackage{extramath}
\usepackage{ericmath}
\usetikzlibrary{arrows.meta,calc,intersections}
\usetikzlibrary{calc}
\usepackage{biblatex}

\usepackage[normalem]{ulem}
\newtheorem{theorem}{Theorem}[section]
\newtheorem{definition}[theorem]{Definition}
\newtheorem{proposition}[theorem]{Proposition}
\newtheorem{lemma}[theorem]{Lemma}
\newtheorem{corollary}[theorem]{Corollary}

\newcommand{\RR}{\mathbb{R}}

\newcommand{\aff}{\mathrm{aff}}
\newcommand{\conv}{\text{conv}}
\newcommand{\affmat}{\mathrm{aff}^{\mathrm{mat}}}
\newcommand{\grlift}[2]{\mathrm{gr}_{#1}(#2)}

\newcommand{\Addresses}{{
    \bigskip
    \footnotesize

    \noindent\textsc{Department of Mathematics, University of Florida, Gainesville, FL, United States}\par\nopagebreak \textit{E-mail address}: \texttt{ericevert@ufl.edu}

    \medskip \medskip
    
    \noindent \textsc{Department of Mathematics, University of Florida, Gainesville, FL, United States}\par\nopagebreak \textit{E-mail address}: \texttt{jack.graham@ufl.edu}
 }}

\begin{document}

\title{Graded face lifts and free extreme points of free spectrahedra}

\author{Eric Evert, Jack Graham}

\begin{abstract}
A free spectrahedron is the matricial solution set of a free linear matrix inequality $L_A(X) = I-A_1 \otimes X_1 - \dots - A_g \otimes X_g \succeq 0$. In this dimension-free setting, free extreme points play the role of classical extreme points. In particular, every bounded real free spectrahedron is the matrix convex hull of its free extreme points. In this qualitative sense, free extreme points of free spectrahedra are abundant. However, quantifications of this abundance have remained elusive. In particular, outside simplices, it is not known whether there exist bounded real free spectrahedra that have finitely many free extreme points. A necessary condition for having finitely many free extreme points is that the classical spectrahedron defined by $L_A(x) \succeq 0$ is a polytope. We strengthen this necessary condition through two constructions. First, we provide a geometric construction of an infinite family of free extreme points at the second level of the maximal matrix convex set over a polygon with at least four sides. Second, we develop a graded face lifting technique for general bounded real free spectrahedra, which allows us to construct free extreme points of the full spectrahedron from free extreme points of the graded face lift. As corollaries, we obtain obstructions to minimal matrix convex sets over polytopes being free spectrahedra and show that much of the scalar extreme-point geometry of maximal matrix convex sets over polytopes can be captured by higher-level free extreme points.
\end{abstract}

\maketitle


\vspace*{.1 in} 


\tableofcontents

\thispagestyle{empty}

\section{Introduction}

Matrix convex sets and their extreme points play an important role in the theory of operator systems \cite{Wit84,EW97,WW99,KPTT13,HL21,DK25} and free semialgebraic geometry \cite{HM12,HKM16,dOHMP09}. In fact, the canonical example of a (compact) matrix convex set is the set of unital completely positive maps on an operator system \cite{WW99}. In this context, the extreme points form a natural collection that completely norms the operator system \cite{DK15}. In another direction, matrix convex sets and their extreme points play a role in quantum information \cite{BN18,BEKMN+,CDN20,CN21}.

In contrast to classical convex sets, there are several natural notions of extreme points for matrix convex sets. One particularly important class is free extreme points, which are closely connected to Arveson's notion of a boundary representation \cite{A69,KLS14}. In general, it is known that a compact matrix convex set can fail to have free extreme points altogether \cite{Evert17,K+,Pas22}; however, in the case of bounded real free spectrahedra, i.e.,  bounded matrix convex sets defined by a real free linear matrix inequality, free extreme points always exist \cite{E_EXTREMEPOINTS}. In fact, it was shown in \cite{EH19} that a Krein-Milman type theorem holds for free extreme points of free spectrahedra. That is, every bounded real free spectrahedron is equal to the matrix convex hull of its free extreme points.

The results  of \cite{E_EXTREMEPOINTS,EH19} imply that free extreme points of bounded real free spectrahedra are plentiful. However, exactly constructing free extreme points is computationally challenging, and a more precise quantification of this abundance has remained elusive \cite{EEHK22}. For example, about ten years ago it was shown that if $K$ is a free spectrahedron whose first level is a simplex, then $K$ has finitely many free extreme points. In fact, the free extreme points of $K$ are precisely the classical extreme points of the first level of $K$ \cite{FNT17,PSS18}. However, since then, it has remained open whether any other bounded real free spectrahedra have finitely many free extreme points modulo unitary equivalence. An equivalent formulation asks which bounded real free spectrahedra have the property that their free polar dual is also a free spectrahedron \cite{E_EXTREMEPOINTS}. 

A partial result in this direction is that every classical extreme point at the first level  of a real free spectrahedron is a free extreme point \cite{E_EXTREMEPOINTS}; thus, if $K$ is a bounded free spectrahedron with finitely many free extreme points, its first level must be a polytope, hence a polygon in two variables. Inspired by \cite{E_QUADRILATERAL}, which studies maximal matrix convex sets over quadrilaterals, we give a geometric construction of an infinite family of level-two free extreme points for maximal polygons with at least four sides. Moreover, we develop a general mechanism which we call graded lifts of faces that allows us to transport such constructions to more general free spectrahedra. Unlike the planar construction, this graded-face mechanism does not require maximality: in fact, it applies to every bounded real free spectrahedron and does not require the scalar level to be polyhedral. Informally, given a matrix convex set $K$ and a face $F$ contained in the first level of $K$, we define the graded lift of $F$ with respect to $K$, denoted $\grlift{K}{F}$, to be the intersection of $K$ and the matrix affine hull of $F$. See Definition \ref{def:grlift} for the formal definition. We mention that our graded lifts of faces are related to, though distinct from, other notions of faces for matrix convex sets which were considered in \cite{KKM+v1} and \cite{KS22}. See Section \ref{sec:OtherFaces} for further discussion. 

The following is an informal statement of Theorem \ref{theorem:FaceLiftExtreme}. 

{\leftskip=2em
\noindent\textit{Let $K$ be a bounded real free spectrahedron and let $F \subset K(1)$ be a face of the first level of $K$. Then $X \in \grlift{K}{F}$ is a free extreme point of $\grlift{K}{F}$ if and only if $X$ is a free extreme point of $K$.}
\par}

A quick corollary of our results is that if $P \subset \R^h$ is a bounded polytope and $v \in P$ is a vertex of $P$ that lies in a nonsimplicial two-dimensional face, then $v$ is not a crucial free extreme point of the maximal matrix convex set over $P$, see Corollary \ref{cor:NoCrucialFree}. Intuitively, this says that if every vertex of $P$ lies in a nonsimplicial face, then the level 1 geometry of the maximal matrix convex set over $P$ is fully captured by higher-dimensional extreme points. We point the reader to Section \ref{sec:notation} for formal definitions.

We mention that our construction for polygons is closely related to a construction of Netzer which appears in the recent preprint \cite{Netzer26}. Netzer's construction in fact obtains the impressive result that a maximal matrix convex set over a polygon with at least five sides has infinitely many free extreme points at every even level. The techniques used in our polygon construction are related to Netzer's, though Netzer uses a clever rank-one perturbation approach, while our approach highlights the level $1$ geometry. The results of \cite{Netzer26} do imply that a maximal matrix convex set over a bounded polytope must have infinitely many free extreme points provided that the polytope has a two-dimensional face with at least five vertices. However, the construction in three or more variables is not explicit. The differences between the  complementary results of \cite{Netzer26} and the present paper are explained in part by the different underlying questions. On the one hand, \cite{Netzer26} considers subhomogeneity of operator systems, which is connected to the existence of free extreme points at arbitrarily large levels. On the other hand, our work explicitly constructs infinite families of free extreme points, though the extreme points we construct do not have arbitrarily large size. 

\subsection{Guide to the Reader}

Section \ref{sec:notation} introduces our key notation and definitions. Section \ref{sec:PolyExtreme} presents our geometric construction of an infinite family of free extreme points at the second level of a maximal matrix convex set over a polygon with at least four sides. Section \ref{sec:FaceLift} presents graded face lifts, shows that free extreme points of a graded face lift are free extreme points of the full matrix convex set, and gives several corollaries of our constructions. In particular, we show that the minimal matrix convex set over a polytope is rarely a free spectrahedron and that extreme points at level $1$ of a maximal matrix convex set over a polygon are rarely crucial free extreme points. 

\subsection{Acknowledgements}

The authors thank Benjamin Passer for helpful conversations which helped motivate this project. 

The authors used ChatGPT by OpenAI for assistance with exposition and development of some aspects of the mathematical arguments. The authors independently verified all mathematical content and take full responsibility for any errors. 

\section{Notation and definitions}
\label{sec:notation}

We follow the notation and definitions of \cite{EPS26}, which gives a thorough introduction to matrix convex sets and their extreme points. We include a brief account of the definitions and notation for the reader's convenience. 

For a fixed $n \in \mathbb{N}$, we let $SM_n (\R)^g$ denote the collection of $g$-tuples $(X_1,\dots,X_g)$ of $n \times n$ symmetric matrices, and we set $SM(\R)^g = \cup_n SM_n (\R)^g$. Similarly, let $M_{n \times m} (\R)^g$ denote $g$-tuples of $n \times m$ matrices. Given a tuple $X \in M_{n \times m}(\R)^g$, we define $X(i,j):=(X_1(i,j),\dots,X_g(i,j))$, where $X_\ell(i,j)$ is the $(i,j)$th entry of $X_\ell \in M_{n \times m} (\R)$. If $X \in M_{n \times 1} (\R)^g$ or $X \in M_{1 \times n} (\R)^g$, then $X(i):=(X_1(i),\dots,X_g(i))$ is defined analogously. Given tuples $X, Y \in SM_n(\R)^g$, we let $X \sim_u Y$ denote that the tuples $X$ and $Y$ are unitarily equivalent. That is, there exists a unitary $U \in M_{n \times n}(\R)$ such that $U^T X U = Y$. Here $U^T X U = (U^T X_1 U,\dots, U^T X_g U)$.

Given a set $K \subset SM(\R)^g$, we let $K(n) = K\cap SM_n(\R)^g$ denote the $n$th-level of $K$. A matrix convex combination of elements of $K$ is a finite sum of the form
\[
\sum_{\ell=1}^m V_\ell^T X^\ell V_\ell \qquad \mathrm{where} \qquad \sum_{\ell=1}^m V_\ell^T V_\ell = I_n.
\]
Here each $X^\ell \in K (n_\ell)$ and $V_\ell \in M_{n_\ell \times n} (\R)$ for each $\ell$. The matrix convex hull of $K$, denoted $\comat(K)$, is the set of matrix convex combinations of elements of $K$. Say $K$ is matrix convex if $K= \comat(K)$. 

Given a closed convex set $C \subset \R^g$, we let $\wmin(C)$ and $\wmax(C)$ denote the minimal and maximal matrix convex sets over $C$. Informally, these are the unique smallest and largest matrix convex sets over $C$ with respect to inclusion. Formally, $\wmin(C) = \comat(C)$, while $\wmax(C)$ is the collection of all tuples $X \in SM(\R)^g$ that satisfy all affine inequalities that are satisfied by $C$. See \cite{PSS18} for further discussion.

A point $X \in K(n)$ is a free extreme point of a matrix convex set $K$ if whenever $X$ is expressed as a matrix convex combination
 \[
 X = \sum_{\ell=1}^m V_\ell^T Y^\ell V_\ell \qquad \mathrm{with} \qquad  \sum_{\ell=1}^m V_\ell^T V_\ell = I_n \qquad \mathrm{and} \qquad V_\ell \neq 0 \ \mathrm{for \ each \ } \ell, 
 \]
 then for each $\ell$, either $n_\ell = n$ and $X \sim_u Y^\ell$ or $n_\ell > n$ and there exists a $Z^\ell \in K$ such that $X \oplus Z^\ell \sim_u Y^\ell$. Denote the set of free extreme points of $K$ by $\free(K)$. We say $X \in \free K$ is a crucial free extreme point of $K$ if 
 \[
 X \notin \overline{\comat(\free K\backslash\{Y \in K: Y \sim_u X\})}.
 \]
 That is, a crucial free extreme point $X$ cannot be expressed as a limit of matrix convex combinations of free extreme points of $K$ that are not unitarily equivalent to $X$. See also \cite{Passer_2019}. Say $X$ is an Arveson  extreme point of $K$ if 
 \[
\begin{pmatrix}
    X & \beta \\
    \beta^T & \gamma
\end{pmatrix} \in K \qquad \mathrm{implies} \ \beta = 0. 
 \]
 We let $\arv K$ denote the set of Arveson extreme points of $K$. By \cite[Theorem 1.1]{E_EXTREMEPOINTS}, $X \in K$ is a free extreme point of $K$ if and only if $X$ is irreducible and $X$ is an Arveson extreme point of $K$. See also \cite[Theorem 1.2]{EH19} for a statement over the reals.

 Say $K \subset SM(\R)^g$ is closed if $K(n)$ is closed for each $n$. Say $K$ is bounded if there exists a $C > 0$ such that $CI - \sum X_i^2 \succeq 0$ for all $X \in K$. It is well known that if $K$ is matrix convex, then $K$ is bounded if and only if $K(1)$ is bounded, see, e.g., \cite{PSS18}.

 Given a tuple $A \in SM_d (\R)^g$ and a tuple $X \in SM_n (\R)^g$, we let $L_A: SM(\R)^g \to SM(\R)$ denote the free monic linear pencil whose evaluation is defined by
 \[
L_A (X) = I_{dn} -A_1 \otimes X_1 -\dots - A_g \otimes X_g. 
 \]
 In the above, $\otimes$ denotes the Kronecker tensor product. The free spectrahedron $\cD_A$ is the matricial solution set of a free linear matrix inequality $L_A(X) \succeq 0$. That is
 \[
\cD_A = \{X \in SM(\R)^g : L_A (X) \succeq 0\}. 
 \]
 It is straightforward to check that a free spectrahedron is matrix convex. We will make frequent use of the set 
 \[
 \mathcal{C}:=\{X\in SM(\mathbb{R})^2:-I \preceq X_1 \preceq I\text{ and }-I \preceq X_2 \preceq I\}.
 \]
 It is straightforward to verify that $\mathcal{C}$ is a free spectrahedron. In fact, $\mathcal{C}$ is the maximal matrix convex set over the unit square. 

 We will also make use of projective maps of free spectrahedra, as defined in \cite{E_QUADRILATERAL}. We omit a formal definition, as we shall only require properties of these maps that are proved in \cite{E_QUADRILATERAL} together with the fact that the restriction of such a map to the first level of a free spectrahedron is indeed a projective map in the classical sense, see, e.g., \cite{BV04}.  

 We emphasize that the matrix convex sets in this article are considered over the reals. The underlying field plays a significant role in the theory of matrix convex sets and their extreme points \cite{Pas22,BM25II,EP25}.

\subsection{Auxiliary Facts}

In this subsection, we collect several known facts that are useful to have stated explicitly. For the reader's convenience, proofs of these facts are included in the appendix. 

We first give a detailed characterization of the family of ellipses inscribed in $\mathcal{C}(1)$. 

\begin{proposition}\label{prop:EllipseFacts}
    The set of ellipses inscribed in the unit square $\mathcal{C}(1)$ forms a one-parameter family. Taking $0<\theta<\pi$ as the parameter, so that $(1,\cos\theta)$ is a point of tangency with the square, the equation for each ellipse is
    \[x^2-2\cos(\theta)xy+y^2=\sin^2(\theta).\]
    
    The major and minor axes are subsets of the diagonals of the square, and the lengths of the two semi-axes are $\sqrt{2}\cos\left(\frac{\theta}{2}\right)$ and $\sqrt{2}\sin\left(\frac{\theta}{2}\right)$. In particular, $\theta$ may be chosen such that the major axis is contained in a particular diagonal of the square, and the length of the minor semi-axis is arbitrarily small.
\end{proposition}
\begin{proof}
See Appendix~\ref{sec:EllipseFactsProof}. 
\end{proof}

Our next proposition states that the matrix convex hull of an irreducible tuple $X \in SM_2(\R)^2$ is always a maximal matrix convex set over an ellipse.

\begin{proposition}
\label{prop:Maximald2}
 Let $X \in SM_2 (\R)^2$ be an irreducible tuple. Then $\comat(X)(1)$ is an ellipse and $\comat(X) = \wmax(\comat(X)(1))$.
\end{proposition}
\begin{proof}
    See Appendix \ref{sec:Maximald2Proof}.
\end{proof}

We next give a geometric refinement of the characterization of the free extreme points of $\mathcal{C}$ given in \cite{E_EXTREMEPOINTS}. 

\begin{proposition}\label{prop:SquareExtreme}
  The free square has an infinite family of nonunitarily equivalent free extreme points at its second level. Moreover, $X \in SM_2(\R)^2$ is a free extreme point of the free square $\mathcal{C}$ at level 2 if and only if $\comat(X)(1)$ is an ellipse inscribed in the square. In this case, $\comat(X)$ is the maximal matrix convex set over this inscribed ellipse.
\end{proposition}
\begin{proof}
    See Appendix \ref{sec:SquareExtremeProof}.
\end{proof}

We end the section with a lemma showing that projective maps of bounded free spectrahedra preserve convex hulls at the first level.

\begin{lemma}\label{lem:projconvcomb}
    Let $P$ be an invertible projective map defined on a bounded free spectrahedron $\cD_A \subset SM(\R)^h$, and let $\{x^1,\dots,x^k\} \subset \cD_A(1).$ Then 
    \[
    P(\conv(\{x^1,\dots,x^k\})) = \conv(\{P(x^1),\dots,P(x^k)\}).
    \]
    Additionally, $P$ is continuous on $\cD_A(1)$. 
\end{lemma}
\begin{proof}
    Since a projective map on a free spectrahedron is a classical projective map on the first level, the result follows from \cite[Section 2.3.3]{BV04}.
\end{proof}

\section{Level $2$ extreme points of maximal polygons}
\label{sec:PolyExtreme}

We will show that, for a bounded polygon $\mathcal{P}$ with at least 4 sides, there exists a one-parameter family of free extreme points at level 2 of $\mathcal{W}^{\max}(\mathcal{P})$. To do this, we will make use of the geometry of inscribed ellipses and the projective equivalence of quadrilaterals. More specifically, we will extend $\mathcal{P}$ to a sufficiently nice bounded quadrilateral $\mathcal{Q}$ and then use an invertible projective transformation to map $\mathcal{W}^{\max}(\mathcal{Q})$ to the free square. This mapping is motivated by the fact that invertible projective maps respect free extreme points and that the square has infinitely many free extreme points at level $2$. Using the fact that, for the free square, level 1 inscribed ellipses correspond to level 2 free extreme points, we can argue that these free extreme points map back to free extreme points of $\mathcal{W}^{\max}(\mathcal{P})$.

Additionally, our particular construction will enable us to express level 1 extreme points of $\wmax(\mathcal{P})$ as limits of matrix convex combinations of free extreme points. Hence, it will show that level 1 extreme points are not crucial free extreme points.

First, we will briefly review the terminology of convex polygons and polar duality, which will be used in our construction. A bounded convex polygon $\mathcal{P}\subseteq\mathbb{R}^2$ is an intersection of finitely many half-planes. More formally, let $\ell_1,...,\ell_d$ be affine linear functionals on $\mathbb{R}^2$, that is, for $x\in\mathbb{R}^2$,
\[
\ell_j(x)=1-\sum_{k=1}^2\alpha_k^{(j)} x_k
\]
for $j=1,...,d$ and $\alpha_k^{(j)}\in\mathbb{R}$. Then
\[
\mathcal{P}=\{x\in\mathbb{R}^2:\ell_j(x)\geq 0\text{ for }j=1,\dots,d\}
\]
is a convex polygon with $0$ in its interior. We will always assume that $\mathcal{P}$ is bounded. In the bounded setting, the defining equations are irredundant if $ \alpha^{(j)} \notin \conv(\{\alpha^{(k)}\}_{k \neq j})$ for each $j = 1,\dots, d$. Assuming that $\{\ell_j\}_{j=1}^d$ is irredundant, $\mathcal{P}$ is a convex polygon with $d$ sides, commonly called a $d$-gon. If the $\{\ell_j\}$ are redundant, then $\mathcal{P}$ has strictly less than $d$ sides, and can be represented by an irredundant collection of strictly less than $d$ affine linear functionals. For this reason, we always assume the defining set $\{\ell_j\}_{j=1}^d$ is irredundant. 

Note that, if one defines $A\in SM_d(\mathbb{R})^2$ by 
\[A=\left(\begin{pmatrix}
    \alpha_1^{(1)} & & & \\
    & \alpha_1^{(2)} & & \\
    & & \ddots & \\
    & & & \alpha_1^{(d)}
\end{pmatrix},\begin{pmatrix}
    \alpha_2^{(1)} & & & \\
    & \alpha_2^{(2)} & & \\
    & & \ddots & \\
    & & & \alpha_2^{(d)}
\end{pmatrix}\right),\]
then $\cD_A(1) = \mathcal{P}$. In fact, $\cD_A = \wmax(\mathcal{P})$. Furthermore, any $d$-gon with $0$ in its interior can be expressed in this manner.

We say that two defining equations $\ell_i$ and $\ell_j$ of $\mathcal{P}$ with $i \neq j$ are adjacent if there exists an extreme point $x\in \mathcal{P}$ such that $\ell_i(x)=0=\ell_j(x)$. We also say that two extreme points $x^i,x^j\in \mathcal{P}$ with $x^i \neq x^j$ are adjacent if there exists a defining affine functional $\ell_k$ of $\mathcal{P}$ such that $\ell_k(x^i)=0=\ell_k(x^j)$.

We recall that the classical polar dual $\mathcal{P}^\bullet$ of a set $\mathcal{P}\subseteq\mathbb{R}^2$ is
\[
\mathcal{P}^\bullet:=\{x\in\mathbb{R}^2 : \ell_y(x) \geq 0 \ \mathrm{for \ all\ } y \in \mathcal{P}\}.
\]
Here, $\ell_y(x) := 1-\langle x,y \rangle$. It is routine to verify that if $\mathcal{P}$ is a bounded  $d$-gon with extreme points $y^1,\dots,y^d$, one has
\[
\mathcal{P}^\bullet = \{x\in\mathbb{R}^2 : \ell_{y^j}(x) \geq 0 \ \mathrm{for \ all\ } j = 1,\dots,d \}.
\]
Furthermore, if $0$ is in the interior of $\mathcal{P}$, then $\mathcal{P}^\bullet$ is a bounded $d$-gon with $0$ in its interior. Additionally, in this case $(\mathcal{P}^\bullet)^\bullet = \mathcal{P}$. In fact, a routine check shows that a closed convex set is bounded if and only if $0$ is in the interior of its dual. 

Next, we will show that any polygon with at least five sides can be extended into a quadrilateral in a way that preserves a pair of opposite corners. Such an extension will eventually allow us to construct an infinite family of arbitrarily narrow inscribed ellipses.

\begin{lemma}\label{quadextension}
    Let $\mathcal{P}$ be a bounded, convex $d$-gon defined by $\ell_1,...,\ell_d$, with 0 in its interior, and $d\geq5$. Let $\hat{x}\in \mathcal{P}$ be an extreme point of $\mathcal{P}$. Without loss of generality, assume $\ell_1$ and $\ell_2$ satisfy $\ell_1(\hat{x})=0=\ell_2(\hat{x})$.  Then there exist some adjacent $\ell_i$ and $\ell_j$ not equal to $\ell_1$ or $\ell_2$, such that the quadrilateral
    \[\mathcal{Q}:=\{x\in\RR^2:\ell_k(x)\geq0\text{ for }k=1,2,i,j\}\]
    is bounded.
\end{lemma}
\begin{proof}
    Given $\ell_1$ and $\ell_2$, we need to choose two other adjacent defining equations $\ell_i,\ell_j$ such that all four form a bounded quadrilateral. We will accomplish this by considering the polar dual of $\mathcal{P}$. Since $\mathcal{P}$ is a convex, bounded $d$-gon with 0 in its interior, $\mathcal{P}^\bullet$ is also a bounded, convex $d$-gon with 0 in its interior. Let $p^1,...,p^d$ be the extreme points of $\mathcal{P}^\bullet$ such that $p^1$ and $p^2$ correspond to $\ell_1$ and $\ell_2$. That is, $\ell_1(x)=1-\langle x,p^1\rangle$ and $\ell_2(x)=1-\langle x,p^2\rangle$. We can easily check that $p^1$ and $p^2$ are adjacent: 
    \[\ell_{\hat{x}}(p^1)=1-\langle p^1,{\hat{x}}\rangle=1-\langle \hat{x},p^1\rangle=\ell_{1}(\hat{x})=0,\]
    and likewise $\ell_{\hat{x}}(p^2)=0$. We must choose two other adjacent extreme points of $\mathcal{P}^\bullet$, such that the convex hull of all four contains the origin in its interior. In that case, the dual of their convex hull will be the desired bounded quadrilateral.

    It is straightforward to show that, if a point $p\in \mathcal{P}^\bullet$ satisfies $\ell_{\hat{x}}(p) = 0$, then there exists a separate point $q\in\mathcal{P}^\bullet$ and a defining equation $\hat{\ell}\neq \ell_{\hat{x}}$ of $\mathcal{P}^\bullet$ such that $\hat{\ell}(q)=0$ and $0 \in \text{conv}(\{p,q\})$. Moreover, without loss of generality, we can choose $p$ so that $p^1 \neq p \neq p^2$ and so that $q$ is not an extreme point of $\mathcal{P}^\bullet$.
    
    Since $q$ is not an extreme point, there exist two adjacent extreme points $q^1,q^2$ such that $q \in \text{conv}(\{q^1,q^2\})$. Since $0 \in \text{conv}(\{p,q\})$, at least one of $q^1$ and $q^2$ is not equal to $p^1$ or $p^2$. If both are not, then let $\mathcal{Q}=(\text{conv}(\{p^1,p^2,q^1,q^2\}))^\bullet$. If one is, say $q^1=p^1$, then let $q^3\neq q^1$ be the other extreme point adjacent to $q^2$. We know that $q^3\neq p^2$, since otherwise $\mathcal{P}^\bullet$ would be a triangle, and consequently $\mathcal{P}$ would have been a triangle. In this case, let $\mathcal{Q}=(\text{conv}(\{p^1,p^2,q^2,q^3\}))^\bullet$.
    
    In either case, $\mathcal{Q}$ is the dual of a quadrilateral that contains 0 in its interior, so $\mathcal{Q}$ is a bounded quadrilateral. Moreover, the two defining equations of $\mathcal{Q}$ that are not $\ell_1$ or $\ell_2$ are adjacent to each other, since they are defined by adjacent extreme points of the dual. Hence $\mathcal{Q}$ has the desired properties.
\end{proof}

We next give a simple lemma which shows that free extreme points of a matrix convex set $K$ remain free extreme if they are elements of a matrix convex subset of $K$.\

\begin{lemma}
\label{lem:ExtremeOfSubset}
Let $K_1 \subseteq  K_2$ be matrix convex sets and suppose $ X \in K_1$ and that $X \in \free K_2$. Then $X \in \free K_1$. 
\end{lemma}
\begin{proof}
    Since $K_1 \subseteq K_2$, any matrix convex combination of elements of $K_1$ is also a matrix convex combination of elements of $K_2$. Thus, if $X$ can be expressed as a nontrivial matrix convex combination of elements of $K_1$, it is clear that the same is true in $K_2$. That is, if $X$ is not free extreme in $K_1$, then $X$ is not free extreme in $K_2$.
\end{proof}

Our next theorem constructs an infinite family of free extreme points at the second level of a maximal matrix convex set over a convex polygon with at least four sides. 

\begin{theorem}\label{thm:PolygonExtreme}
    Let $\mathcal{P}$ be a bounded, convex $d$-gon with $d\geq4$. Then $\mathcal{W}^{\max}(\mathcal{P})$ has infinitely many free extreme points at level $2$ modulo unitary equivalence.
\end{theorem}
\begin{proof}
    It is straightforward to check that free extreme points and maximality of matrix convex sets are preserved by invertible free affine linear transformations, so we can without loss of generality assume that $\mathcal{P}$ has $0$ in its interior. Thus, $\wmax(\mathcal{P})$ is a free spectrahedron.

    The $d=4$ case follows from \cite{E_QUADRILATERAL}. Still, the idea is useful for our upcoming construction, so we recall it here. If $d=4,$ then $\wmax(\mathcal{P})$ is a free quadrilateral. By \cite[Theorem 4.1]{E_QUADRILATERAL}, there exists an invertible projective map $P$ between $\mathcal{W}^{\max}(\mathcal{P})$ and the free square, such that $P(\mathcal{W}^{\max}(\mathcal{P}))=\mathcal{C}$. By \cite[Theorem 3.7]{E_QUADRILATERAL}, we have that invertible projective maps respect Arveson extreme points of bounded free spectrahedra. In particular, if $X\in\mathcal{C}(2)$ is Arveson extreme in $\mathcal{C}$, then $P^{-1}(X)\in\wmax(\mathcal{P})(2)$ is Arveson extreme in $\wmax(\mathcal{P})$.
    Since $P$ is invertible, and there exist infinitely many free extreme points in $\mathcal{C}(2)$ by Proposition \ref{prop:SquareExtreme}, there exist infinitely many Arveson extreme points in $\mathcal{W}^{\max}(\mathcal{P})(2)$, which are the preimages of the free (and hence Arveson) extreme points in $\mathcal{C}(2)$. It is routine to verify that an invertible projective map between two bounded free spectrahedra preserves irreducibility. Consequently, as the preimages of irreducible Arveson extreme points, these points are all irreducible, hence free extreme. Thus $\mathcal{W}^{\max}(\mathcal{P})$ has infinitely many free extreme points at level $2$.

    Now suppose $d\geq 5$. Let $B$ be an extreme point of $\mathcal{P}$, and let $\ell_1,...,\ell_d$ be the defining affine functionals of $\mathcal{P}$. Without loss of generality, assume $\ell_1$ and $\ell_2$ satisfy $\ell_1(B)=0=\ell_2(B)$. By Lemma \ref{quadextension}, there exist two adjacent defining equations $\ell_i$ and $\ell_j$ such that the quadrilateral $\mathcal{Q}=\{x:\ell_k(x)\geq0\text{ for }k=1,2,i,j\}$ is bounded. Like before, by \cite[Theorem 4.1]{E_QUADRILATERAL}, there exists an invertible projective map $P$ between $\mathcal{W}^{\max}(\mathcal{Q})$ and the free square, such that $P(\mathcal{W}^{\max}(\mathcal{Q}))=\mathcal{C}$. Furthermore, since $B$ is an extreme point of $\mathcal{Q}$, the projective map $P$ may be chosen so $P(B) = (1,1)$. We will show that there exist infinitely many ellipses in $P(\mathcal{P})$ corresponding to the matrix ranges of free extreme points of $\mathcal{C}$.
    
    To describe the properties of $P(\mathcal{Q})$ and $P(\mathcal{P})$, it will be useful to label all relevant extreme points of $\mathcal{P}$ and $\mathcal{Q}$, as in Figure \ref{thm:PolygonExtremepic}. Let $E$ denote the extreme point of $\mathcal{Q}$ that is not adjacent to $B$. Without loss of generality, we can assume that $\ell_1,\ell_i$ are adjacent in $\mathcal{Q}$ and also that $\ell_2,\ell_j$ are adjacent in $\mathcal{Q}$. Let $G$ and $H$ denote the other two extreme points of $\mathcal{Q}$ such that $\ell_1(G)=0=\ell_i(G)$ and $\ell_2(H)=0=\ell_j(H)$.
    Let $A$ denote the extreme point of $\mathcal{P}$ adjacent to $B$ such that $\ell_1(A)=0$, and likewise let $C$ denote the extreme point of $\mathcal{P}$ adjacent to $B$ such that $\ell_2(C)=0$. Similarly, let $D$ denote the extreme point adjacent to $E$ satisfying $\ell_j(D) = 0$, and $F$ the extreme point adjacent to $E$ satisfying $\ell_i(F) =0$. For the sake of notation, we let $A':=P(A)$, and likewise for the other points.

    \begin{figure}
    \centering
    \begin{tikzpicture}[
    >=Stealth,
    every node/.style={font=\small},
    dot/.style={circle,fill,inner sep=2pt}
]
        
        \begin{scope}[scale=0.60]
        
        \coordinate (F) at (0,1.3);
        \coordinate (X1) at (0.8,2.2);
        \coordinate (A) at (2.0,2.35);
        \coordinate (B) at (4.2,1.3);
        \coordinate (C) at (4.45,0.1);
        \coordinate (X2) at (3.5,-1.7);
        \coordinate (X3) at (2.4,-2.0);
        \coordinate (D) at (0.75,-1.55);
        \coordinate (E) at (0,-0.75);
        
        \draw[thick]
        (F)--(X1)--(A)--(B)--(C)--(X2)--(X3)--(D)--(E)--cycle;
        
        \foreach \P/\Pos/\Lab in {
        A/above/A,
        B/right/B,
        C/right/C,
        D/below/D,
        E/left/E,
        F/left/F}
        {
            \fill (\P) circle (3.3pt);
            \node[\Pos] at (\P){$\Lab$};
        }
        
        
        \coordinate (Q1) at ($(A)!1.7!(B)$);
        
        \coordinate (Q2) at ($(F)!-1.3!(E)$);
        
        \coordinate (G) at (0,3.3);
        
        \draw[dashed,thick] (A)--(G);
        \draw[dashed,thick] (F)--(G);
        
        \coordinate (H) at (5.8,-7);
        
        \draw[dashed,thick] (C)--(H);
        \draw[dashed,thick] (D)--(H);

        \foreach \P/\Pos/\Lab in {
        G/above/G,
        H/right/H}
        {
            \fill (\P) circle (3.3pt);
            \node[\Pos] at (\P){$\Lab$};
        }
        
        \end{scope}
        
        
        \draw[->,thick]
        (4,1.05)
        to[bend left=28]
        node[above] {$P$}
        (7,1.05);
        
        
        \begin{scope}[shift={(9.0,0)}]
        
        
        \def\t{0.75}
        
        \def\Atop{-0.10}
        \def\Fleft{0.10}
        \def\Dbot{-0.15}
        \def\Cright{0.18}
        
        
        \coordinate (NW) at (-1,1);
        \coordinate (NE) at (1,1);
        \coordinate (SE) at (1,-1);
        \coordinate (SW) at (-1,-1);
        
        
        \coordinate (Ap) at (\Atop,1);
        \coordinate (Bp) at (1,1);
        \coordinate (Cp) at (1,\Cright);
        \coordinate (Dp) at (\Dbot,-1);
        \coordinate (Ep) at (-1,-1);
        \coordinate (Fp) at (-1,\Fleft);
        \coordinate (Hp) at (1,-1);
        \coordinate (Gp) at (-1,1);
        
        
        \draw[thick] (Ep)--(Fp);
        \draw[thick] (Ep)--(Dp);
        \draw[thick] (Ap)--(Bp);
        \draw[thick] (Bp)--(Cp);
        
        \draw[dashed,thick] (Fp)--(NW)--(Ap);
        \draw[dashed,thick] (Cp)--(SE)--(Dp);
        
        
        \foreach \P in {Ap,Bp,Cp,Dp,Ep,Fp,Gp,Hp}
            \fill (\P) circle (2pt);
        
        \node[above] at (Ap) {$A'$};
        \node[above right] at (Bp) {$B'$};
        \node[right] at (Cp) {$C'$};
        \node[below] at (Dp) {$D'$};
        \node[below left] at (Ep) {$E'$};
        \node[left] at (Fp) {$F'$};

        \node[above left] at (Gp) {$G'$};
        \node[below right] at (Hp) {$H'$};
        
        
        \draw[red,dashed,thick] (Fp)--(Ap);
        \draw[red,dashed,thick] (Dp)--(Cp);
        
        
        \coordinate (TR) at (1,\t);
        
        \pgfmathsetmacro{\aa}{sqrt((1+\t)/2)}
        \pgfmathsetmacro{\bb}{sqrt((1-\t)/2)}
        
        \begin{scope}
        \clip (-1,-1) rectangle (1,1);
        
        \draw[red,very thick,samples=250,smooth,variable=\x,domain=0:360]
        plot ({\aa*cos(\x)-\bb*sin(\x)},
              {\aa*cos(\x)+\bb*sin(\x)});
        \end{scope}
        
        \end{scope}
        
        \end{tikzpicture}
    \caption{Mapping $\mathcal{Q}$ to $\mathcal{C}(1)$}
    \label{thm:PolygonExtremepic}
\end{figure}
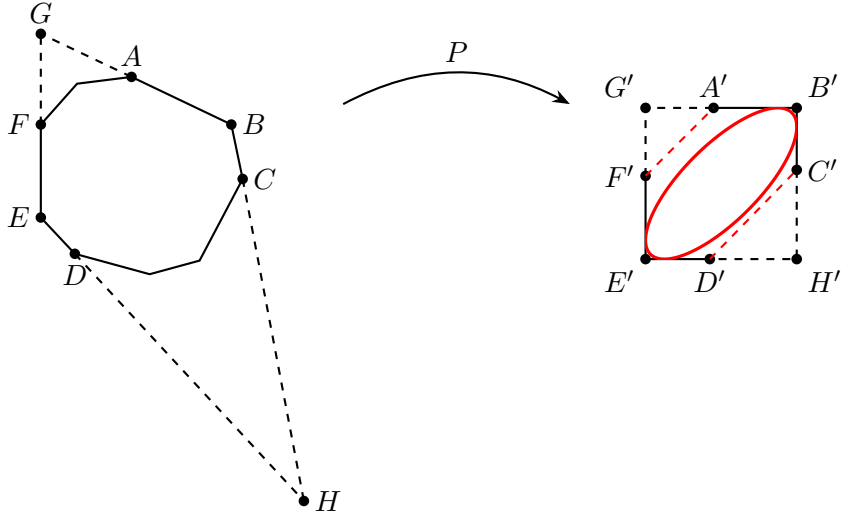   

    We will now verify that $P$ preserves the relevant geometry of $\mathcal{P}$ and $\mathcal{Q}$. By Lemma \ref{lem:projconvcomb}, we have that projective maps respect convex hulls. As a consequence, they map extreme points to extreme points, as well as lines to lines. So, $B'$ and $E'$ must be extreme points of the square.
    Suppose $B'$ and $E'$ are adjacent. Then $\text{conv}(\{B',E'\})$ is a subset of the boundary of $\mathcal{C}(1)$. Let $X'\in \text{conv}(\{B',E'\})$ be a boundary point, with $B'\neq X'\neq E'$. It follows that the point $X:=P^{-1}(X')$ lies on $\text{conv}(\{B,E\})\subseteq \mathcal{Q}$. Since $B\neq X\neq E$, and $B$ and $E$ are not adjacent, we have that $X$ lies in the interior of $\mathcal{Q}$. Thus there exists some open ball $B_{\epsilon}(X)\subseteq \mathcal{Q}$. Since $P^{-1}$ is also an invertible projective map, and invertible projective maps are continuous on their domain, $B_{\epsilon}(X)$ must map to an open set $P(B_{\epsilon}(X))\subseteq \mathcal{C}(1)$. But this is impossible because $X'$ is on the boundary of $\mathcal{C}(1)$. Thus, $E'$ and $B'$ are not adjacent. It follows that $G'$ and $H'$ form the other nonadjacent extreme point pair of the square.
    
    Since $A\in \text{conv}(\{B,G\})$, we have that $A'\in \text{conv}(\{B',G'\})$. Likewise, $C', D',$ and $F'$ are each on their corresponding separate sides of the square. Without loss of generality, suppose that $A'$ has the shortest distance to the diagonal $\text{conv}(\{B',E'\})$.   Let $\delta=\text{dist}(A',\text{conv}(\{B',E'\}))$.

    Recall that we have chosen $P$ so that $B'=(1,1)$. By Proposition \ref{prop:EllipseFacts}, we can choose a point $(1,\cos\theta)$ on the side of $\mathcal{C}(1)$ sufficiently close to $B'$ that determines an ellipse $\mathcal{E}_\theta\subseteq \mathbb{R}^2$ inscribed in $\mathcal{C}(1)$, with the properties that $(1,\cos\theta)\in\mathcal{E}_\theta$, the major axis is contained in the diagonal $\text{conv}(\{B',E'\})$, and the length of the minor semi-axis is less than $\delta$. It is now straightforward to show that 
    \[
    \mathcal{E}_\theta\subseteq \text{conv}(\{A',B',C',D',E',F'\})\subseteq P(\mathcal{P}).
    \]
    Hence, $\mathcal{E}_\theta$ is an ellipse in $P(\mathcal{P})$ that is also inscribed in the square. By Proposition \ref{prop:SquareExtreme}, there exists a free extreme point $X_\theta\in\mathcal{C}(2)$ such that $\comat(X_\theta)(1)=\mathcal{E}_\theta$. As a consequence, we have that
    \[
    X_\theta\in \wmax(\mathcal{E}_\theta)\subseteq \wmax(P(\mathcal{P})).
    \]
    Furthermore, since $\mathcal{C}$ is the maximal matrix convex set over the unit square, we have $\mathcal{W}^{\max}(P(\mathcal{P}))\subseteq\mathcal{C}$. Since $X_\theta$ is free extreme in $\mathcal{C}$, it follows from Lemma \ref{lem:ExtremeOfSubset} that $X_\theta$ is free extreme in $\wmax(P(\mathcal{P}))$.
    
    We now argue that $\mathcal{W}^{\max}(P(\mathcal{P}))=P(\mathcal{W}^{\max}(\mathcal{P}))$. Since level one of $P(\wmax(\mathcal{P}))$ is $P(\mathcal{P})$, we immediately have that 
    \[\wmax(P(\mathcal{P}))\supseteq P(\wmax(\mathcal{P})).\]
    For the other direction, observe that level one of $P^{-1}(\wmax(P(\mathcal{P})))$ is $\mathcal{P}$. Then we have 
    \[P^{-1}(\wmax(P(\mathcal{P})))\subseteq\wmax(\mathcal{P}),\]
    and consequently
    \[\wmax(P(\mathcal{P}))\subseteq P(\wmax(\mathcal{P})).\]
    Thus, $\mathcal{W}^{\max}(P(\mathcal{P}))=P(\mathcal{W}^{\max}(\mathcal{P}))$, hence $X_\theta$ is free extreme in $P(\mathcal{W}^{\max}(\mathcal{P}))$.
    
    As in the $d=4$ case, \cite[Theorem 3.7]{E_QUADRILATERAL} lets us conclude that $P^{-1}(X_\theta)$ is Arveson extreme in $\mathcal{W}^{\max}(\mathcal{P})$. Furthermore, since $X_\theta$ is irreducible, $P^{-1}(X_\theta)$ is also irreducible, and hence a free extreme point of $\mathcal{W}^{\max}(\mathcal{P})$ at level $2$. It is clear that the above construction holds for any $\phi \in (0,\theta]$, hence $\{P^{-1}(X_\phi)\}$ is an infinite family of nonunitarily equivalent free extreme points of $\wmax(\mathcal{P})$, which completes the proof.
\end{proof}

The previous construction also provides us with a method to show that the level one extreme points of $\wmax(\mathcal{P})$ are not crucial free extreme points.

\begin{corollary}
\label{cor:NoCrucialG2}
    Let $\mathcal{P}$ be a bounded, convex $d$-gon with $d\geq 4$. Then the extreme points of $\mathcal{P}$ are not crucial free extreme points of $\mathcal{W}^{\max}(\mathcal{P})$. 
\end{corollary}
\begin{proof}

    WLOG assume that $0$ is in the interior of $\mathcal{P}$. Let $Y\in \mathcal{P}$ be an extreme point of $\mathcal{P}$.
    Following the proof of Theorem \ref{thm:PolygonExtreme}, we construct the quadrilateral $\mathcal{Q}$ and obtain an invertible projective map $P$ from $\wmax(\mathcal{Q})$ to $\mathcal{C}$. As in the proof of Theorem \ref{thm:PolygonExtreme}, we can assume that $P(Y)=(1,1)$.
    Furthermore, by letting $\phi$ approach $0$ in the proof above, we can choose an infinite family of ellipses $\{\mathcal{E}_j\}_{j=1}^{\infty}$ such that  $\mathcal{E}_j\subseteq P(\mathcal{P})$ for all $j$, and such that $(1,1)\in\overline{\conv(\cup_j\mathcal{E}_j)}$. For each $j$, let $X^j\in SM_2(\RR)^2$ be an irreducible tuple in $P(\wmax(\mathcal{P}))$ such that $\comat(X^j)(1)=\mathcal{E}_j$. As argued in the proof of Theorem \ref{thm:PolygonExtreme}, $X^j$ is free extreme in $P(\wmax(\mathcal{P}))$. Additionally, since $\mathcal{E}_j = \comat(X^j)(1)$, we have that $(1,1)\in \overline{\comat(\{X^j\})}$. It follows that $Y\in \overline{\comat(\{P^{-1}(X^j)\})}$. Furthermore, $Y \in \R^2$ cannot be unitarily equivalent to $P^{-1}(X^j) \in SM_2(\R)^2$ for any $j$. Finally, by \cite[Theorem 3.7]{E_QUADRILATERAL}, we have that $P^{-1}(X^j)$ is free extreme in $\wmax(\mathcal{P})$ for all $j$. We conclude that $Y$ is not crucial free extreme in $\wmax(\mathcal{P})$.
    \end{proof}

\section{Extreme points of graded lifts of faces}

\label{sec:FaceLift}

Fix positive integers $1 \leq g<h$. In this section, we show that free extreme points of a free spectrahedron $\cD_A \subset SM(\R)^h$ can be obtained from free extreme points of the graded lift of a $g$-dimensional face of $\cD_A(1)$. Recall that a subset $F$ of a convex set $C\subset \R^h$ is said to be a face if, whenever $x,y \in C$ and there exists  a $t \in (0,1)$ such that $tx + (1-t)y \in F$, one has $x,y \in F$, see, e.g., \cite{RG95}.

To streamline our proofs, it will be convenient to have a standard position for the faces we consider. To this end, let $C \subset \R^h$ be a closed bounded convex set with $0$ in its interior, and let $F$ be a $g$-dimensional face of $C$. Letting $\aff(F)$ denote the affine hull of $F$, we say $F$ is in standard position relative to $C$ if the following assumptions hold:
\begin{enumerate}
    \item $\aff(F) = \{x \in \R^h: x_{g+1} = 1 \ \mathrm{and} \ x_j = 0 \mathrm{ \ for \ all \ } j>g+1\}$.
    \item $(0,\dots,0,1,0,\dots,0)$ is in the relative interior of $F$. 
    \item $C \subseteq \{x: x_{g+1} \leq 1\}.$
\end{enumerate}

\begin{lemma}
\label{lem:StandardPos}
 Let $C \subset \R^h$ be a closed bounded convex set containing $0$ in its interior and let $F \subset C$ be a $g$-dimensional face of $C$. Then there exists an invertible linear transformation $T:\R^h \to \R^h$ such that $T(F)$ is in standard position relative to $T(C)$.  
\end{lemma}
\begin{proof}
 The proof of Lemma \ref{lem:StandardPos} follows from standard convex geometric techniques. 
\end{proof}

We now aim to define the graded lifts  of faces at level $1$ of a free spectrahedron. To this end, let $C$ be a closed convex set and $F \subset C$ be a $g$-dimensional face of $C$. We define the matrix affine hull of $F$, denoted $\affmat(F),$ by
\[
\affmat(F) = \comat (\aff(F)).
\]

As $\aff(F)$ has codimension $h-g$, there exists a collection of $h-g$ affine linear functionals whose common zero set is equal to $\aff(F)$. The following  proposition shows that $\affmat(F)$ is equal to the matricial zero set of any collection of affine linear functionals that define $\aff(F)$. 

\begin{proposition}
    Let $C \subset \R^h$ be a closed bounded convex set containing $0$ in its interior and let $F \subset C$ be a $g$-dimensional face of $C$. Let $\ell_1,\dots,\ell_{h-g}$ be a collection of affine linear functionals with the property that
    \[
    \aff(F) = \{x \in \R^h : \ell_j (x) = 0 \mathrm{\ for \ all \ } j=1,\dots, h-g\}. 
    \]
    Then 
    \[
    \affmat(F) = \{X \in SM(\R)^h : \ell_j (X) = 0 \mathrm{\ for \ all \ } j=1,\dots, h-g\}. 
    \]
\end{proposition}
\begin{proof}
        Using Lemma \ref{lem:StandardPos}, we can WLOG assume that $F$ is in standard position with respect to $C$. 
    For notational convenience, define
    \[
    \mathcal{F} :=\{X \in SM(\R)^h : \ell_j (X) = 0 \mathrm{\ for \ all \ } j=1,\dots, h-g\}.
    \]
    It is straightforward to verify that matrix convex combinations preserve affine relations, from which it follows that $\affmat(F) \subset \mathcal{F}$. 

    For the reverse direction, since $F$ is in standard position with respect to $C$, up to a standard change of basis on the defining affine functionals, we can WLOG assume $\ell_1(x) = 1-x_{g+1}$ and $\ell_j(x) = x_{g+j}$ for $j=2,\dots,h-g$. Thus we have
    \[
    \aff(F) =\{(y,1,0,\dots,0) \in \R^h : y \in \R^g\}, \]
    and, moreover,
    \[
    \mathcal{F} = \{(Y,I,0,\dots,0) \in SM(\R)^h: Y\in SM(\R)^g\}. 
    \]
    It is well known that any tuple $Y \in SM(\R)^g$ can be expressed as a matrix convex combination of a collection of tuples $y^1,\dots,y^m \in \R^g$, see, e.g., \cite{DDSS17,HKMS19}. It follows that $\mathcal{F} \subset \affmat(F)$, which completes the proof. 
\end{proof}

We now define the graded lift of a face at the first level of a matrix convex set.

\begin{definition}
\label{def:grlift}
Let $K \subset SM(\R)^h$ be a closed matrix convex set, and let $F \subset K(1)$ be a face of $K(1)$. Define the graded lift of the face $F$ with respect to $K$ by
\[
\mathrm{gr}_K(F) = \affmat(F) \cap K. 
\]
\end{definition}

It is not hard to see that one always has 
\[
F = \mathrm{gr}_K(F) (1).
\]
Furthermore, a straightforward check shows that if $T$ is an invertible linear transformation of $\R^h$, then 
\[
T(\grlift{K}{F}) = \grlift{T(K)}{T(F)}. 
\]

\begin{theorem}
\label{theorem:FaceLiftExtreme}
Suppose $\cD_A\subseteq SM(\R)^h$ is a bounded real free
spectrahedron and suppose that $\cD_A(1)$ has a $g$-dimensional face
$F$ that is in standard position with respect to $\cD_A(1)$. Define
\[
K
=
\{
X\in SM(\R)^g :
(X,I,0,\dots, 0)\in \grlift{\cD_A}{F}\}.
\]
Then $K$ is a bounded free spectrahedron. Moreover,
\[
X\in \partial^{\mathrm{free}}K \quad\Longleftrightarrow\quad
(X,I,0,\dots,0)\in \partial^{\mathrm{free}}\grlift{\cD_A}{F}
\quad\Longleftrightarrow\quad 
(X,I,0,\dots,0)\in \partial^{\mathrm{free}}\mathcal D_A.
\]
\end{theorem}

\begin{proof}
To simplify the notation in the proof, we use $(X,I,0)$ to denote $(X,I,0,\dots,0)$. Observe that $(X,I,0)$ is irreducible if and only if $X$ is irreducible. Thus, using \cite[Theorem 1.2]{EH19} it is sufficient to show that $X$ is Arveson extreme in $K$ if and only if $(X,I,0)$ is Arveson extreme in $\cD_A$. Additionally, by assumption $F$ is nonempty, so $K$ is also nonempty. Furthermore, $K$ is bounded since $\cD_A$ is bounded.

 We first show that
$K$ is a free spectrahedron. To this end, define $\check{A} = (A_1,\dots,A_g)$ and observe that by definition
$X\in K$ if and only if
\begin{align*}
L_A(X,I,0)
&=
I_d\otimes I_n
-A_1\otimes X_1
-\cdots
-A_g\otimes X_g
-A_{g+1}\otimes I_n
-\sum_{j=g+2}^h A_j\otimes 0\\
&=
(I_d-A_{g+1})\otimes I_n
-\Lambda_{\check A}(X)
\succeq 0.
\end{align*}

Now, our assumptions imply that $0$ is in the interior of $K(1)$, from which it follows that
\[
\ker(I-A_{g+1}) \subseteq
\ker\Lambda_{\check A}(x)
\]
for all $x \in \R^g$, hence
\[
\ker(I-A_{g+1}) \subseteq \ker A_i
\qquad \mathrm{for \ all \ }
i\leq g.
\]
Additionally, we obtain $I-A_{g+1} \succeq 0$. 

Letting $\dagger$ denote the Moore-Penrose pseudoinverse, it follows that
\[
(I-A_{g+1})^{\dagger/2}
\bigl(
I-A_{g+1},A_1,\ldots,A_g
\bigr)
(I-A_{g+1})^{\dagger/2}
\sim_u
\left(
\begin{pmatrix}
I & 0\\
0 & 0
\end{pmatrix},
\begin{pmatrix}
E_1 & 0\\
0 & 0
\end{pmatrix},
\ldots,
\begin{pmatrix}
E_g & 0\\
0 & 0
\end{pmatrix}
\right),
\]
for some $E = (E_1,\dots,E_g) \in SM(\R)^g$. 
We conclude that
\[
X\in K
\quad \mathrm{if \ and \ only \ if} \quad
L_E(X)\succeq0.
\]
That is, $K$ is equal to the free spectrahedron $\cD_E$. 

We now assume that $X \in SM_n (\R)^g$ is Arveson extreme in $K = \cD_E$. Using \cite[Lemma 2.1]{EH19}, to show that $(X,I,0) \in \cD_A$ is Arveson extreme in $\cD_A$, it is sufficient to show that if 
\[
\beta
=
(\beta_1,\ldots,\beta_{g+1},\beta_{g+2},\ldots,\beta_h) \in M_{n \times 1} (\R)^h,
\]
satisfies
\[
\ker L_A(X,I,0)
\subseteq
\ker\Lambda_A(\beta^T),
\]
then $\beta = 0$. To this end, set
\[
Y=(X,I,0).
\]
Again using \cite[Lemma 2.1]{EH19}, there exists a scalar $c>0$ and a $\gamma \in \R^h$ such that
\[
\begin{pmatrix}
Y & c\beta\\
c\beta^T & \gamma
\end{pmatrix}
\in\mathcal D_A.
\]
Hence, by choosing appropriate isometric restrictions to level $2$,
we obtain
\[
\begin{pmatrix}
Y(i,i) & c\beta(i)\\
 c\beta(i) & \gamma
\end{pmatrix}
\in\mathcal D_A(2),
\qquad
i=1,\ldots,n.
\]
It follows that
\[
\ker L_A(Y(i,i))
\subseteq
\ker\Lambda_A(c\beta(i))
=
\ker\Lambda_A(\beta(i)).
\]
As argued in the proof of \cite[Corollary 2.3]{E_EXTREMEPOINTS}, this kernel containment implies that there exists an $\epsilon > 0$ such that $Y(i,i) \pm \epsilon \beta(i) \in \cD_A(1)$. Furthermore, $Y(i,i) = (X(i,i),1,0)\in F$, so the facial structure of $F$ implies
\[
\beta_j(i)=0
\qquad
\text{for all } j\geq g+1.
\]
Since this holds for all $i=1,\ldots,n$, we obtain
\[
\beta_j=0
\qquad
\text{for } j\geq g+1.
\]

Now, set $\check{\beta} := (\beta_1,\dots,\beta_g)$. The containment
\[
\ker(I-A_{g+1})\subseteq\ker A_i \qquad \mathrm{for \ all \ }
i\leq g,
\]
together with the positivity of $L_A(X,I,0)$ allows us to conclude that
\[
\ker L_A(X,I,0)
\subseteq
\ker\Lambda_A(\check{\beta}^T,0),
\]
if and only if
\begin{align*}
\ker \left(L_{E}(X) \oplus 0\right)
&=
\ker\left[
\bigl((I-A_{g+1})^{\dagger/2}\otimes I_n\bigr)
L_A(X,I,0)
\bigl((I-A_{g+1})^{\dagger/2}\otimes I_n\bigr)
\right]\\
&\subseteq
\ker\left[
\bigl((I-A_{g+1})^{\dagger/2}\bigr)
\Lambda_A(\check{\beta}^T,0)
\bigl((I-A_{g+1})^{\dagger/2}\otimes I_n\bigr)
\right]\\
&=
\ker\left(\Lambda_{E}(\check{\beta}^T) \oplus 0\right).
\end{align*}
Thus
\beq
\label{eq:FaceKer}
\ker L_E(X)
\subseteq
\ker\Lambda_E(\check{\beta}^T).
\eeq
Finally, since
\[
X\in\partial^{\mathrm{Arv}}K = \partial^{\mathrm{Arv}}\cD_E,
\]
we can use \cite[Lemma 2.1]{EH19} with equation~\eqref{eq:FaceKer} to conclude that
$
\check{\beta}=0.
$
It follows that $\beta = 0$, from which we conclude that
\[
(X,I,0)\in\partial^{\mathrm{Arv}}(\mathcal D_A).
\]

It remains to show the reverse implication. To this end, assume that $X\in K (n)$ is not Arveson extreme in $K$. Then there exists some $0 \neq \beta \in M_{n \times 1} (\R)^g$ and $\gamma \in \R^g$ such that
\[
\begin{pmatrix}
X & \beta\\
\beta^T & \gamma
\end{pmatrix}
\in K, \qquad \mathrm{hence} \qquad
\left(
\begin{pmatrix}
X & \beta\\
\beta^T & \gamma
\end{pmatrix},
\begin{pmatrix}
I & 0\\
0 & 1
\end{pmatrix},
\begin{pmatrix}
0 & 0\\
0 & 0
\end{pmatrix}
\right)
\in\mathcal D_A.
\]
It follows that $(X,I,0)$ is not Arveson extreme in $\cD_A$. We conclude that
\[
X\in \partial^{\mathrm{free}}K
\quad\Longleftrightarrow\quad 
(X,I,0)\in \partial^{\mathrm{free}}\mathcal D_A.
\]

It remains to show that 
\[
X\in \partial^{\mathrm{free}}K \quad\Longleftrightarrow\quad
(X,I,0)\in \partial^{\mathrm{free}}\grlift{\cD_A}{F}.
\]
As in the first part of the proof, it is straightforward to show that $ (X,I,0)\in \partial^{\mathrm{free}}\grlift{\cD_A}{F}$ implies $X \in \free K$. On the other hand, if $X \in \free K$, then, as shown in the first part of the proof, $(X,I,0) \in \free \cD_A$. Since $(X,I,0) \in \grlift{\cD_A}{F} \subset \cD_A$, it follows from Lemma \ref{lem:ExtremeOfSubset} that $(X,I,0) \in \free \grlift{\cD_A}{F}.$
\end{proof}

\begin{corollary}
\label{cor:FaceLiftExtreme} 
Suppose $\cD_A\subseteq SM(\R)^h$ is a bounded real free
spectrahedron, and let $F$ be a $g$-dimensional face of $\cD_A(1)$. Then $X \in \grlift{\cD_A}{F}$ is a free extreme point of $\grlift{\cD_A}{F}$ if and only if $X$ is a free extreme point of $\cD_A$. 
\end{corollary}
\begin{proof}
    It is straightforward to show that invertible linear transformations preserve free extreme points of matrix convex sets. Thus, the result follows from Lemma \ref{lem:StandardPos} together with Theorem~\ref{theorem:FaceLiftExtreme}.
\end{proof}

\begin{corollary}
\label{cor:NoCrucialFree}
    Let $\mathcal{P} \subset \R^g$ be a closed bounded polytope. Suppose $\mathcal{P}$ has a two-dimensional face $F$ that is not a simplex. Then the extreme points of $F$ are not crucial free extreme points of $\mathcal{W}^{\max}(\mathcal{P})$. Furthermore, $\wmax(\mathcal{P})$ has infinitely many free extreme points modulo unitary equivalence. As a consequence, if $0$ is in the interior of $\mathcal{P}$, then $\wmin(\mathcal{P}^\bullet)$ is not a free spectrahedron. 
\end{corollary}
\begin{proof}
    It is straightforward to check that $\grlift{\wmax(\mathcal{P})}{F} = \wmax(F)$. Thus, the fact that extreme points of $F$ are not crucial free extreme points follows immediately from Corollary \ref{cor:FaceLiftExtreme} together with Corollary \ref{cor:NoCrucialG2}. The fact that $\wmax(\mathcal{P})$ has infinitely many nonunitarily equivalent free extreme points follows from Theorem \ref{thm:PolygonExtreme} with Corollary \ref{cor:FaceLiftExtreme}. Finally, the fact that $\wmin(\mathcal{P}^\bullet)$ is not a free spectrahedron follows from the fact that $\wmax(\mathcal{P})$ has infinitely many free extreme points modulo unitary equivalence together with \cite[Theorem 1.2]{E_EXTREMEPOINTS}. 
\end{proof}

\subsection{Other notions of faces of matrix convex sets}
\label{sec:OtherFaces}
To conclude, we mention that various notions of faces of matrix convex sets have been considered in other works, including \cite{KKM+v1,KS22}. Our graded lifts of a face are not the same as the faces considered in either of these works. Indeed, the faces considered in \cite{KS22} lie at a fixed matrix level and are better suited to matrix extreme points than free extreme points. In particular, \cite[Proposition 4.4]{KS22} shows that a $C^*$ extreme point of such a face is a matrix extreme point of the matrix convex set. 

On the other hand, the first arXiv version of \cite{KKM+v1} considers dimension-free faces of matrix convex sets. However, our graded lifts can fail to be free faces in this sense. A notable contrast here is \cite[Example 9.9]{KKM+v1}, which shows that a free extreme point of a free face can fail to be a free extreme point of the full matrix convex set. This striking difference may be related to the difference between free faces and graded lifts of faces; however, it may also be related to the difference between real free spectrahedra and general matrix convex sets. As previously discussed, the extreme points of the former are much better behaved than the extreme points of the latter \cite{E_EXTREMEPOINTS,EH19,K+,Pas22}.

\printbibliography

@article{ACL_ELLIPSES,
title = {Duality and Inscribed Ellipses},
journal = {Computational Methods and Function Theory},
volume = {15},
pages = {635--644},
year = {2015},
author = {Agarwal, M. and Clifford, J. and Lachance, M.},
}

@article{E_QUADRILATERAL,
   title={The Arveson boundary of a free quadrilateral is given by a noncommutative variety},
   ISSN={1846-3886},
   number={15},
   journal={Operators and Matrices},
   publisher={Element d.o.o.},
   author={Evert, Eric},
   year={2021},
   pages={1351–1378}
}

@article{E_EXTREMEPOINTS,
   title={Extreme Points of Matrix Convex Sets, Free Spectrahedra, and Dilation Theory},
   volume={28},
   ISSN={1559-002X},
   number={2},
   journal={The Journal of Geometric Analysis},
   publisher={Springer Science and Business Media LLC},
   author={Evert, Eric and Helton, J. William and Klep, Igor and McCullough, Scott},
   year={2018},
   month=May, pages={1373–1408}
}

@article{KS22,
title = {Facial structure of matrix convex sets},
journal = {Journal of Functional Analysis},
volume = {283},
number = {7},
pages = {109601},
year = {2022},
issn = {0022-1236},
author = {Igor Klep and Tea Štrekelj},
}

@misc{KKM+v1,
  author        = {Kennedy, Matthew and Kim, Se-Jin and Manor, Nicholas},
  title         = {Nonunital operator systems and noncommutative convexity},
  year          = {2021},
  eprint        = {2101.02622v1},
  archivePrefix = {arXiv},
  primaryClass  = {math.OA}
}

@article{Passer_2019,
   title={Compressions of compact tuples},
   volume={564},
   ISSN={0024-3795},
   journal={Linear Algebra and its Applications},
   publisher={Elsevier BV},
   author={Passer, Benjamin and Shalit, Orr Moshe},
   year={2019},
   month=Mar, pages={264–283}
}

@article{ A69  ,
	author       = { W. Arveson },
	title        = { Subalgebras of ${C}^*$-algebras },
	year         = { 1969 },
	journal      = { Acta Math },
	volume       = { 123 },
	pages        = { 141-224 }
}

@article{ DK15 ,
	author       = { K.R. Davidson and M. Kennedy },
	title        = { The {C}hoquet boundary of an operator system },
	year         = { 2015 },
	journal      = { Duke Math. J. },
	volume       = { 164 },
	pages        = { 2989-3004 }
}

@book{DK25,
  author    = {Davidson, Kenneth R. and Kennedy, Matthew},
  title     = {Noncommutative Choquet Theory},
  series    = {Memoirs of the American Mathematical Society},
  volume    = {316},
  number    = {1608},
  publisher = {American Mathematical Society},
  year      = {2025}
}

@article{ DDSS17 ,
	author       = { K.R. Davidson and A. Dor-On and M. O. Shalit and B. Solel },
	title        = { Dilations, inclusions of matrix convex sets, and completely positive maps },
	journal      = { Int. Math. Res. Not. },
	volume       = { 13 },
	year         = { 2017 },
	pages        = { 4069-4130 }
}

@article{BM25II,
author = {Blecher, David P. and McClure, Caleb Becker},
title = {Real Noncommutative Convexity. II: Extremality and Noncommutative Convex Functions},
journal = {Mathematische Nachrichten},
volume = {299},
number = {6},
pages = {1280-1305},
year = {2026}
}

@article{ EW97 ,
	author       = { E.G. Effros and S. Winkler },
	title        = { Matrix convexity: operator analogues of the bipolar and {H}ahn-{B}anach theorems },
	journal      = { J. Funct. Anal. },
	volume       = { 144 },
	year         = { 1997 },
	pages        = { 117-152 }
}

@article{Evert17,
	label = {  Evert18  },
	author       = { E. Evert },
	title        = { Matrix convex sets without absolute extreme points },
	journal      = { Linear Algebra Appl. },
	volume       = { 537 },
	year         = { 2018 },
	pages        = { 287-301 }
}

@article{  EH19 ,
	author       = { E. Evert and J.W. Helton },
	title        = { {A}rveson extreme points span free spectrahedra },
	journal      = { Math. Ann. },
	volume       = { 375 },
	year         = { 2019 },
	pages        = { 629-653 }
}

@Inbook{EPS26,
author="Evert, Eric
and Passer, Benjamin
and {\v{S}}trekelj, Tea",
editor="Alpay, Daniel
and Colombo, Fabrizio
and Sabadini, Irene",
title="Extreme Points of Matrix Convex Sets and Their Spanning Properties",
bookTitle="Operator Theory",
year="2026",
publisher="Springer Nature Switzerland",
address="Cham",
pages="2643--2675",
}

@Article{EEHK22,
  author={Epperly, A. and Evert, E. and  Helton, J.W. and Klep, I.},
title = {Matrix extreme points and free extreme points of free spectrahedra},
journal = {Optimization Methods and Software},
volume = {39},
number = {6},
pages = {1263--1308},
year = {2024},
}

@book{BV04,
  author    = {Stephen Boyd and Lieven Vandenberghe},
  title     = {Convex Optimization},
  publisher = {Cambridge University Press},
  year      = {2004}
}

@article{EP25,
title = {Matrix convex sets over the Euclidean ball and polar duals of real free spectrahedra},
journal = {Linear Algebra and its Applications},
volume = {736},
pages = {54-83},
year = {2026},
author = {Eric Evert and Benjamin Passer}
}

@incollection{Wit84,
title = {On Matrix Order and Convexity},
editor = {Klaus-Dieter Bierstedt and Benno Fuchssteiner},
series = {North-Holland Mathematics Studies},
publisher = {North-Holland},
volume = {90},
pages = {175-188},
year = {1984},
booktitle = {Functional Analysis: Surveys and Recent Results III},
author = {Gerd Wittstock}
}

@article{FNT17 ,
	author       = { T. Fritz and T. Netzer and A. Thom },
	title = {Spectrahedral Containment and Operator Systems with Finite-Dimensional Realization},
    journal = {SIAM Journal on Applied Algebra and Geometry},
    volume = {1},
    number = {1},
    pages = {556-574},
    year = {2017}
}

@article{HL21, 
	author       = { M. Hartz and M. Lupini },
	title        = { Dilation theory in finite dimensions and matrix convexity },
	journal      = { Isr. J. Math. },
	volume       = { 245 },
	year         = { 2021 },
	pages        = { 39-73 }
}

@article{ HM12 ,
	author       = { J.W. Helton and S. McCullough },
	title        = { Every free basic convex semi-algebraic set has an {LMI} representation },
	journal      = { Ann. of Math. },
	volume       = { 176 },
	year         = { 2012 },
	pages        = { 939-1013 },
	number       = { 2 },
}

@article{ HKM16 ,
	author       = { J.W. Helton and I. Klep and S. McCullough },
	title        = { Matrix Convex hulls of free semialgebraic sets },
	journal      = { Trans. Amer. Math. Soc. },
	volume       = { 368 },
	year         = { 2016 },
	pages        = { 3105-3139 }
}

@article{BEKMN+,
  title={Inclusion constants for free spectrahedra with applications to quantum incompatibility},
  author={Bluhm, A. and Evert, E. and Klep, I. and Magron, V. and Nechita, I.},
  journal =	 {arXiv},
  year = {2025},
  url = {https://arxiv.org/abs/2512.17706},
}

@article{KPTT13,
  author  = {Kavruk, Ali S. and Paulsen, Vern I. and Todorov, Ivan G. and Tomforde, Mark},
  title   = {Quotients, exactness, and nuclearity in the operator system category},
  journal = {Advances in Mathematics},
  volume  = {235},
  year    = {2013},
  pages   = {321--360}
}

@article{ KLS14  ,
	author       = { C. Kleski },
	title        = { Boundary representations and pure completely positive maps },
	journal      = { J. Operator Theory },
	volume       = { 71 },
	year         = { 2014 },
	pages        = { 45-62 }
}

@article{ K+  ,
	author       = { Kriel, TL },
	title        = { An Introduction to Matrix Convex Sets and Free Spectrahedra },
	journal      = { Complex Anal. Oper. Theory },
	volume       = { 13 },
	year         = { 2019 },
	pages        = { 3251-3335 }
}

@article{Netzer26,
  author  = {Netzer, Tim},
  title   = {Subhomogeneity and Arveson Boundary of Free Polyhedra},
  journal = {arXiv preprint arXiv:2608.04158},
  year    = {2026}
}

@incollection{dOHMP09,
	author       = {M. de Oliveira and J.W. Helton and S. McCullough and M. Putinar},
	title        = {Engineering systems and free semi-algebraic geometry},
	year         = {2009},
	booktitle    = {Emerging applications of algebraic geometry},
	publisher    = {Springer-Verlag},
	pages        = {17-61},
	editor       = {M. Putinar and S. Sullivant}
}

@article{PSS18,
title = {Minimal and maximal matrix convex sets},
journal = {Journal of Functional Analysis},
volume = {274},
number = {11},
pages = {3197-3253},
year = {2018},
author = {B. Passer and O. Shalit and B. Solel}
}

@article{ RG95 ,
	author       = { M. Ramana and A.J. Goldman },
	title        = { Some geometric results in semidefinite programming },
	journal      = { J. Global Optim. },
	volume       = { 7 },
	year         = { 1995 },
	pages        = { 33-50 }
}

@article{ WW99 ,
	author       = { C. Webster and S. Winkler },
	title        = { The {K}rein-{M}ilman Theorem in Operator Convexity },
	journal      = { Trans Amer. Math. Soc. },
	volume       = { 351 },
	year         = { 1999 },
	pages        = { 307-322 }
}

@article {Pas22,
    AUTHOR = {Passer, B.},
     TITLE = {Complex free spectrahedra, absolute extreme points, and
              dilations},
   JOURNAL = {Doc. Math.},
  FJOURNAL = {Documenta Mathematica},
    VOLUME = {27},
      YEAR = {2022},
     PAGES = {1275--1297},
      ISSN = {1431-0635},
   MRCLASS = {47A20 (46L07 47A13)},
  MRNUMBER = {4466985},
}

@article {HKMS19,
    AUTHOR = {Helton, J.W. and Klep, I. and McCullough, S. and
              Schweighofer, M.},
     TITLE = {Dilations, linear matrix inequalities, the matrix cube problem
              and beta distributions},
   JOURNAL = {Mem. Amer. Math. Soc.},
  FJOURNAL = {Memoirs of the American Mathematical Society},
    VOLUME = {257},
      YEAR = {2019},
    NUMBER = {1232},
     PAGES = {vi+106},
      ISSN = {0065-9266},
      ISBN = {978-1-4704-3455-7; 978-1-4704-4947-6},
   MRCLASS = {47A20 (33B15 46L07 60E05 90C22)},
  MRNUMBER = {3898991}
}

@article {BN18,
    AUTHOR = {Bluhm, A. and Nechita, I.},
     TITLE = {Joint measurability of quantum effects and the matrix diamond},
   JOURNAL = {J. Math. Phys.},
  FJOURNAL = {Journal of Mathematical Physics},
    VOLUME = {59},
      YEAR = {2018},
    NUMBER = {11},
     PAGES = {112202, 27},
      ISSN = {0022-2488},
   MRCLASS = {81P15},
  MRNUMBER = {3880587}
}

@article {CDN20,
    AUTHOR = {De las Cuevas, G. and Drescher, T. and Netzer, T.},
     TITLE = {Quantum magic squares: dilations and their limitations},
   JOURNAL = {J. Math. Phys.},
  FJOURNAL = {Journal of Mathematical Physics},
    VOLUME = {61},
      YEAR = {2020},
    NUMBER = {11},
     PAGES = {111704, 15},
      ISSN = {0022-2488},
   MRCLASS = {15B51 (16S50 81P45)},
  MRNUMBER = {4174414}
}

@article{CN21,
  author  = {De les Coves, G. and Netzer, T.},
  title   = {Quantum Information Theory and Free Semialgebraic Geometry:
             One Wonderland Through Two Looking Glasses},
  journal = {Internationale Mathematische Nachrichten},
  year    = {2021},
  number  = {246},
  pages   = {1--29}
}

\clearpage
\markboth{ERIC EVERT, JACK GRAHAM}
{GRADED FACE LIFTS AND FREE EXTREME POINTS OF FREE SPECTRAHEDRA}
\appendix
\section{}

The appendix contains proofs of Propositions \ref{prop:EllipseFacts}, \ref{prop:Maximald2}, and \ref{prop:SquareExtreme}.

\subsection{Proof of Proposition \ref{prop:EllipseFacts}}
\label{sec:EllipseFactsProof}

\begin{proof}
 The fact that there exists a unique family of ellipses inscribed in the square is a consequence of \cite[Theorem 3.1]{ACL_ELLIPSES}. We can take the prescribed point of contact on a single side of the square to be $(1,\cos\theta)$ for $0<\theta<\pi$. We will show that the equation
    \[(\csc^2\theta)x^2+(-2\csc^2\theta\cos\theta)xy+(\csc^2\theta)y^2=1\] describes this unique family. The equation describes a conic centered at 0. A conic with equation $Ax^2+Bxy+Cy^2=1$ is an ellipse if and only if $A>0$ and $4AC-B^2>0$. Plugging our values in, we see this holds, so our equation describes an ellipse. Multiplying both sides by $\sin^2\theta$ shows it is equivalent to the equation in the statement of the proposition. Plugging in $x=1$ and solving for $y$ yields the point $(1,\cos\theta)$, so this touches the square at the prescribed point. Using implicit differentiation, we can find that each ellipse of this form is tangent to the square at four points. Since this equation describes a unique inscribed ellipse for each point of contact on one side of the square, it must be the one-parameter family of inscribed ellipses.

    Plugging in $x=-1, y= 1$, and $y=-1$, we find that the other three tangency points are $(-1,-\cos\theta), (\cos\theta,1)$, and $(-\cos\theta,-1)$. Notice that these are symmetric about the diagonals of the square. If we were to reflect an ellipse about these diagonals, we would get another ellipse with the same tangency points. Since one tangent point is enough to define a unique ellipse inscribed in the square, this is the same ellipse. Thus, the ellipse is symmetric about the diagonals of the square, and so its major and minor axes are subsets of the diagonals of the square.

    We can find the lengths of the semi-axes by first finding what points of the ellipse lie on $y=x$ and $y=-x$. Plugging in $y=x$ and solving for $x$, we find
    \[x^2=\frac{\sin^2\theta}{2-2\cos\theta}.\]
    Simplifying and using a power-reduction identity gives us $\pm(\cos\left(\frac{\theta}{2}\right),\cos\left(\frac{\theta}{2}\right))$ as the intersection points. Likewise, $y=-x$ yields $\pm(\sin\left(\frac{\theta}{2}\right),-\sin\left(\frac{\theta}{2}\right))$. Their distances from 0 are $\sqrt2\cos\left(\frac{\theta}{2}\right)$ and $\sqrt{2}\sin\left(\frac{\theta}{2}\right),$ respectively. Thus these are the lengths of the major and minor semi-axes (depending on which is larger for the given choice of $\theta$).
    \end{proof}

\subsection{Proof of Proposition \ref{prop:Maximald2}}

\label{sec:Maximald2Proof}

\begin{proof}
    Set
    \[
F:= \left(\begin{pmatrix}
        1 & 0 \\
        0 & -1
    \end{pmatrix},\begin{pmatrix}
        0 & 1 \\
        1 & 0
    \end{pmatrix}\right).
\]
     Since $X$ is irreducible, one must have that $\{I_2,X_1,X_2\}$ forms a basis for $SM_2(\R)$, hence there exist parameters $\alpha_0,\alpha_1,\alpha_2$ and $\beta_0,\beta_1,\beta_2$ such that 
    \[
    (\alpha_0 I_2 +\alpha_1 X_1+\alpha_2 X_2, \beta_0 I_2 +\beta_1 X_1+\beta_2 X_2) = F. 
    \]
    Define the map $T:SM(\R)^2 \to SM(\R)^2$ by 
    \[
    T(Y) =  (\alpha_0 I_n +\alpha_1 Y_1+\alpha_2 Y_2, \beta_0 I_n +\beta_1 Y_1+\beta_2 Y_2)
    \]
    for $Y \in SM_n (\R)^2$. Thus $T$ is a free affine linear map in the sense of \cite{E_QUADRILATERAL}. Moreover, one can verify that $T$ is invertible. It is also straightforward to verify that free affine linear maps respect matrix convex combinations. That is, if $\sum_{j=1}^k V_j^T V_j = I$, then
    \[
    T\left(\sum_{j=1}^k V_j^T X^jV_j\right)= \sum_{j=1}^k V_j^T T(X^j) V_j.
    \]
    It follows that
    \[
    \comat(X) =  \comat(T^{-1}(F)) = T^{-1} (\comat(F)) = T^{-1} (\wmax(\mathbb{B}^2)),
    \]
    where $\mathbb{B}^2$ is the closed Euclidean ball. See, e.g., \cite{HKMS19} for the equality $\comat(F)=\wmax(\mathbb{B}^2)$.  It is also straightforward to verify that invertible free affine linear maps preserve maximality of matrix convex sets. We conclude that $\comat(X)(1)$ is an ellipse and that $\comat(X)$ is the maximal matrix convex set over that ellipse. 
\end{proof}

\subsection{Proof of Proposition \ref{prop:SquareExtreme}}
\label{sec:SquareExtremeProof}

\begin{proof}
    
   From \cite{E_EXTREMEPOINTS}, we have that $X=(X_1,X_2)$ is a free extreme point of $\mathcal{C}$ if and only if $X_1^2 = X_2^2 = I$ and $X$ is irreducible. It is straightforward to check that $X \in SM_2(\R)^2$ satisfies these conditions if and only if $X$ is unitarily equivalent to a tuple of the form
    \[
    Y=    \left(\begin{pmatrix}
        1 & 0 \\
        0 & -1
    \end{pmatrix},\begin{pmatrix}
        \cos{\theta} & \sin{\theta} \\
        \sin{\theta} & -\cos{\theta}
    \end{pmatrix}\right)
    \]
    for some $\theta$ such that $\sin(\theta) \neq 0$. Set $E:=\comat(X)(1)$. Proposition \ref{prop:Maximald2} shows that $E$ is an ellipse and that $\comat(X) = \wmax(E)$. Furthermore, $E \subset \mathcal{C}(1)$, since $X \in \mathcal{C}$. To complete the forward direction of the proof, we need only show that $E$ is inscribed in $\mathcal{C}(1)$.

    To this end, observe that by compressing to the top-left and bottom-right corners of $Y$, one has that $(1,\cos(\theta)),(-1,-\cos(\theta)) \in \comat(X).$ Similarly, one has that $Y$ is unitarily equivalent to the tuple 
    \[
    \left(\begin{pmatrix}
        \cos{\theta} & -\sin{\theta} \\
        -\sin{\theta} & -\cos{\theta}
    \end{pmatrix},
      \begin{pmatrix}
        1 & 0 \\
        0 & -1
    \end{pmatrix}\right),
    \]
    so again considering compressions shows $(\cos(\theta),1),(-\cos(\theta),-1) \in \comat(X)$. It follows that $E$ is indeed inscribed in $\mathcal{C}(1)$. 

    For the reverse direction, suppose $X \in SM_2(\R)^2$ and that $\comat(X)(1)$ is an ellipse inscribed in $\mathcal{C}(1)$. Note that this implies that $X$ is irreducible. Indeed, if $X \in SM_2 (\R)^2$ is reducible, then $\comat(X)(1)$ is a line segment. Using Proposition \ref{prop:Maximald2} then shows that $\comat(X)$ is the maximal matrix convex set over $\comat(X)(1)$. 
    
    Now, from Proposition \ref{prop:EllipseFacts}, every ellipse inscribed in the square is determined by a point on the boundary of the form $(1,\cos(\theta))$, and since the matrix convex hulls of irreducible tuples are equal if and only if the tuples are unitarily equivalent, we obtain that $X$ must be unitarily equivalent to 
  \[
 \left(\begin{pmatrix}
        1 & 0 \\
        0 & -1
    \end{pmatrix},\begin{pmatrix}
        \cos{\theta} & \sin{\theta} \\
        \sin{\theta} & -\cos{\theta}
    \end{pmatrix}\right).
    \]
    It follows from \cite{E_EXTREMEPOINTS} that $X$ is a free extreme point of the square. 

    Finally, to complete the proof, note that if $E_1 \neq E_2$ are ellipses inscribed in the square and $X^1$ and $X^2$ are tuples in $SM_2(\R)^2$ such that
    \[
    \comat(X^j)(1) = E_j \qquad \mathrm{for\ } j=1,2,
    \]
 then $X^1$ and $X^2$ cannot be unitarily equivalent. Thus we obtain a natural correspondence between the family of ellipses inscribed in $\mathcal{C}(1)$ and a family of nonunitarily equivalent free extreme points at level $2$ of $\mathcal{C}$. 
\end{proof}

 \Addresses 

\end{document}